\documentclass[a4paper,12pt,dvipdfmx]{article}
\usepackage{amsmath,amssymb,amsthm}
\numberwithin{equation}{section}
\usepackage[margin=20mm]{geometry}

\newtheorem{thm}{Theorem}[section]

\newtheorem{lem}[thm]{Lemma}

\makeatletter

\begin{document}
\title{Berry-Esseen bound for the Brownian motions on hyperbolic spaces\thanks{This work was 
supported by JSPS KAKENHI Grant Numbers JP22K18675, JP23H01076.}}
\author{Yuichi Shiozawa\thanks{Department of Mathematics,
Graduate School of Science, Osaka University, Toyonaka, Osaka, 560-0043,
Japan; \texttt{shiozawa@math.sci.osaka-u.ac.jp}}}
\maketitle

\begin{abstract}
We obtain the uniform convergence rate for the Gaussian fluctuation 
of the radial part of the Brownian motion on a hyperbolic space. 
We also show that this result is sharp 
if the dimension of the hyperbolic space is two or general odd. 
Our approach is based on the repetitive use of the Millson formula and the integration by parts formula. 
\end{abstract}

\section{Introduction}
We are concerned with the Gaussian fluctuation 
of the radial part of the Brownian motion on a $d$-dimensional hyperbolic space. 
In this note, we obtain the uniform convergence rate for the fluctuation in distribution, 
the so-called Berry-Esseen bound, 
together with the sharpness for $d=2$ or general odd $d\ge 3$.  

For $d\ge 2$, let ${\mathbb H}^d$ be the $d$-dimensional hyperbolic space 
with a pole $o$, 
and let $d=d_{{\mathbb H}^d}$ be the associated distance function.
Let $\Delta=\Delta_{{\mathbb H}^d}$ be the Laplace-Beltrami operator, 
and $X=(\{X_t\}_{t\geq 0},\{P_x\}_{x\in {\mathbb H}^d})$ 
the Brownian motion on ${\mathbb H}^d$ generated by $\Delta/2$.
Let $R_t^{(d)}=d(o,X_t) \ (t\geq 0)$ be the radial process and $P=P_o$. 
Then by the It\^o formula applied to $R_t^{(d)}$, 
we have
\begin{equation}\label{eq:radial}
R_t^{(d)}=B_t+\frac{d-1}{2}\int_0^t\coth(R_s^{(d)})\,{\rm d}s,
\end{equation}
where $B_t$ is the Brownian motion on ${\mathbb R}$ 
(see, e.g., \cite[Example 3.3.3]{H02}).
Since $\lim_{t\rightarrow\infty}R_t^{(d)}=\infty$, 
we obtain the law of large numbers:
\begin{equation}\label{eq:linear}
\lim_{t\rightarrow\infty}\frac{R_t^{(d)}}{t}=\frac{d-1}{2}, \quad \text{$P$-a.s.}
\end{equation}
(see also \cite[Subsection 4.1]{GH09} and \cite{S17} for the escape rate). 
Moreover, we realize the limiting behavior of the fluctuation in \eqref{eq:linear} 
as the central limit theorem: 
\begin{equation}\label{eq:clt}
\lim_{t\rightarrow\infty}
P\left(\frac{R_t^{(d)}-(d-1)t/2}{\sqrt{t}}\geq x\right)=\Phi(x), \quad x\in {\mathbb R}
\end{equation}
(see, e.g., \cite[Corollary 3.1]{B94} and \cite[Theorem 2.1]{M10}), 
where 
$$
\Phi(x)=\frac{1}{\sqrt{2\pi}}\int_x^{\infty}e^{-u^2/2}\,{\rm d}u, \quad x\in {\mathbb R}.
$$

In connection with the heat equation in ${\mathbb H}^d$, 
V\'azquez \cite[(5.5)]{V22} noted that 
the limit in \eqref{eq:clt} is uniformly convergent in $x\in {\mathbb R}$. 
Our purpose in this note is to establish the uniform convergence rate in \eqref{eq:clt}:
\begin{thm}
\label{thm:berry}
For any $d\geq 2$, 
there exists a constant $c_1>0$ such that 
\begin{equation}\label{eq:berry-1}
\sup_{x\in {\mathbb R}}
\left|P\left(\frac{R_t^{(d)}-(d-1)t/2}{\sqrt{t}}\geq x\right)-\Phi(x)\right|
\le \frac{c_1}{\sqrt{t}}, \quad t\ge 1.
\end{equation}
Moreover, if $d=2$ or $d\ge 3$ is odd,  
then there exists a constant $c_2>0$ such that 
\begin{equation}\label{eq:berry-2}
P\left(\frac{R_t^{(d)}-(d-1)t/2}{\sqrt{t}}\geq 0\right)-\Phi(0)
\ge \frac{c_2}{\sqrt{t}}, \quad t\ge 1.
\end{equation}
Namely, the convergence rate of \eqref{eq:berry-1} is sharp. 
\end{thm}

Theorem \ref{thm:berry} provides the convergence rate $t^{-1/2}$,  
which is consistent with the standard Berry-Esseen theorem 
for i.i.d.\ random variables (see, e.g., \cite[Theorem 3.4.17]{D19}). 

We explain our approach to Theorem \ref{thm:berry}. 
According to \eqref{eq:radial} and \eqref{eq:linear}, 
we can regard $R_t^{(d)}$ as the Brownian motion with linear drift  asymptotically. 
However, 
we have no information about the convergence rate in distribution of  
\eqref{eq:linear}. 
Even though we also know the matching bound on the transition density function of $X$ 
(see \eqref{eq:heat-bound} below), 
this bound  is insufficient 
for the estimate of the distribution which would be necessary in the proof of \eqref{eq:berry-1}.
Instead of using \eqref{eq:radial} or \eqref{eq:heat-bound}, 
we repeatedly use the Millson formula (see \eqref{eq:millson} below) 
together with the integration by parts formula. 
By the Millson formula, 
we reduce the calculation of the distribution under consideration 
to $d=2$ or $d=3$. 
At present, the validity of \eqref{eq:berry-2} is unavailable for general even dimensions 
because of the difficulty in calculation. 

Concerning the law of large numbers \eqref{eq:linear}, 
Cammarota-De Gregorio-Macci \cite[Proposition 3.2]{CDM14} and Hirao \cite[Theorem 1.1]{H11} 
proved the large deviation principle for the radial part by using \eqref{eq:heat-bound}. 
Moreover, Cammarota-De Gregorio-Macci \cite[Section 3, p.1560--1563]{CDM14} 
proved the moderation deviation principle, 
together with the exponential decay order of the distribution related to the asymptotic normality. 
Theorem \ref{thm:berry} establishes the asymptotic normality in the suitable scaling factor.

Our motivation of this work lies in clarifying 
how the structures of the volume and spectral determine 
the long time behavior of a symmetric Markov process. 
Anker-Setti \cite[Theorem 2]{AS92} revealed the concentration behavior in distribution 
of the Brownian motion on a complete and non-compact Riemannian manifold 
for which the volume growth rate is exponential 
and the bottom of the $L^2$-spectrum of the Laplacian is strictly positive. 
In particular, this result characterizes the linear growth rate of  the radial part of the Brownian motion 
in terms of the exponential volume growth rate and the bottom of the $L^2$-spectrum.
For the Riemannian manifold with a pole, 
Grigor'yan-Hsu \cite[Theorem 4.1]{GH09} determined the linear growth rate exactly. 
Theorem \ref{thm:berry} is an attempt to provide the second order asymptotic behavior of 
the linear growth phenomena of a symmetric Markov process in a quantitative way.

We close this introduction with some words on contents and notations.
The proof of Theorem \ref{thm:berry} is given in Section \ref{sect:berry}. 
A lemma on some elementary calculus is postponed to Appendix.
For a fixed constant $T>0$,  
let $f_1(t)$ and $f_2(t)$ be positive functions defined on $[T,\infty)$. 
We then write $f_1(t)\sim f_2(t)$ if $f_1(t)/f_2(t)\rightarrow 1$ as $t\rightarrow \infty$. 
Let $S$ be a set, and let $g_1(s)$ and $g_2(s)$ be positive functions defined on $S$. 
We then write $g_1(s)\lesssim g_2(s)$ if 
there exists a constant $c>0$ such that 
$g_1(s)\le cg_2(s)$  for all $s\in S$. 
We further write $g_1(s)\asymp g_2(s)$ 
if $g_1(s)\lesssim g_2(s)$ and $g_2(s)\lesssim g_1(s)$.

\section{Proof of Theorem \ref{thm:berry}}\label{sect:berry}
In this section, we first recall the Millson formula 
and the estimates of the transition density function of 
the Brownian motion on ${\mathbb H}^d$. 
We then prove Theorem \ref{thm:berry} 
for odd dimensions and even dimensions in this order. 
 
\subsection{Preliminaries}
Let $d\ge 2$. 
The $d$-dimensional hyperbolic space ${\mathbb H}^d$ is 
a spherically symmetric Riemannian manifold with the Riemannian distance given by 
$$
{\rm d}s^2={\rm d}r^2+(\sinh r)^2\,{\rm d}\theta^2.
$$
Here ${\rm d}\theta^2$ is the distance on the $(d-1)$-dimensional surface 
$S^{d-1}=\{\theta\in {\mathbb R}^d \mid |\theta|=1\}$. 
We write $\omega_d=2\pi^{d/2}/\Gamma(d/2)$ for the surface area of $S^{d-1}$. 
Let $d$ and ${\rm d}v$ denote 
the associated distance function and volume measure, respectively.

Let $X=(\{X_t\}_{t\geq 0},\{P_x\}_{x\in {\mathbb H}^d})$ 
be the Brownian motion on ${\mathbb H}^d$ generated by $\Delta/2$, 
the half of the Laplace-Beltrami operator on ${\mathbb H}^d$.
Then there exists a Borel measurable function 
$p_d(t,x,y):(0,\infty)\times {\mathbb H}^d\times {\mathbb H}^d \to (0,\infty)$ 
such that 
$$
P_x(X_t\in A)=\int_A p_d(t,x,y)\,v({\rm d}y), 
\quad x\in {\mathbb H}^d, \ t>0, \ A\in {\cal B}({\mathbb H}^d).
$$
Namely, $p_d(t,x,y)$ is the transition density function of $X$. 
Moreover, 
there exists a Borel measurable function $q_d(t,r):(0,\infty)\times (0,\infty)\to (0,\infty)$ 
such that $p_d(t,x,y)=q_d(t,d(x,y))$ for any $x,y\in {\mathbb H}^d$ and $t>0$, and 
\begin{equation}\label{eq:dist-radial}
P_x(d(x,X_t)\in B)
=\omega_d \int_B q_d(t,r) \sinh^{d-1}(r) \,{\rm d}r, \quad t>0, \ B\in {\cal B}([0,\infty))
\end{equation}
(see, e.g., \cite[\S 2.2]{CDM14} for details).

It is known that 
\begin{equation}\label{eq:d=2}
q_2(t,r)=\frac{2^{1/2}e^{-t/8}}{(2\pi t)^{3/2}}
\int_r^{\infty}\frac{s e^{-s^2/(2t)}}{(\cosh s-\cosh r)^{1/2}}\,{\rm d}s
\end{equation}
and 
\begin{equation}\label{eq:d=3}
q_3(t,r)=\frac{e^{-t/2}}{(2\pi t)^{3/2}}\frac{r}{\sinh r}e^{-r^2/(2t)}
\end{equation}
(see, e.g., \cite[Section 5.7]{D92} or \cite[Section 2]{G96} and references therein). 
The Millson formula is a recursive relation 
between the transition density functions:
\begin{equation}\label{eq:millson}
q_d(t,r)=-\frac{e^{-(d-2)t/2}}{2\pi \sinh r}\frac{\partial q_{d-2}}{\partial r}(t,r), \quad t>0, \ r>0
\end{equation}
(see, e.g., \cite{DM88, GN98}, \cite[Section 5.7]{D92} and references therein).
We also see by \cite[Theorem 5.7.2]{D92} that 
\begin{equation}\label{eq:heat-bound}
q_d(t,r)
\asymp \frac{1}{t^{d/2}}
\exp\left(-\frac{(d-1)^2}{8}t-\frac{d-1}{2r}-\frac{r^2}{2t}\right)(1+r+t)^{(d-3)/2}(1+r), \quad t>0, \ r>0.
\end{equation}

\subsection{Odd dimensions}
In this subsection, we first prove Theorem {\rm \ref{thm:berry}} for $d=3$. 
Using this assertion, we next prove Theorem {\rm \ref{thm:berry}} for general odd dimensions.
\begin{proof}[Proof of Theorem {\rm \ref{thm:berry}} for $d=3$]
We first prove \eqref{eq:berry-1}. 
We write $R_t=R_t^{(3)}$ for simplicity. 
Since $R_t\geq 0$ for any $t\geq 0$ and $\omega_3=4\pi$,  
it follows by  \eqref{eq:d=3} that for any $x\in {\mathbb R}$, 
\begin{equation*}
\begin{split}
P\left(\frac{R_t-t}{\sqrt{t}}\geq x\right)
=P\left(R_t\geq (t+x\sqrt{t})\vee 0\right)
&=\frac{\omega_3}{(2\pi t)^{3/2}}\int_{(t+x\sqrt{t})\vee 0}^{\infty}
e^{-t/2}e^{-u^2/(2t)} u\sinh u\,{\rm d}u\\
&=\frac{1}{\sqrt{2\pi}}\frac{2e^{-t/2}}{t\sqrt{t}}
\int_{(t+x\sqrt{t})\vee 0}^{\infty}
e^{-u^2/(2t)}u\sinh u\,{\rm d}u.
\end{split}
\end{equation*}
Then by the change of variables formula with $u=t+v\sqrt{t}$, 
we obtain 
\begin{equation*}
\begin{split}
&\frac{2e^{-t/2}}{t\sqrt{t}}
\int_{(t+x\sqrt{t})\vee 0}^{\infty}
e^{-u^2/(2t)}u\sinh u\,{\rm d}u\\
&=2e^{-t}\int_{x\vee (-\sqrt{t})}^{\infty}e^{-v^2/2}e^{-v\sqrt{t}}\sinh(t+v\sqrt{t})\,{\rm d}v
+\frac{2e^{-t}}{\sqrt{t}}\int_{x\vee (-\sqrt{t})}^{\infty}e^{-v^2/2}e^{-v\sqrt{t}}v\sinh(t+v\sqrt{t})\,{\rm d}v\\
&=I_1(t,x)+I_2(t,x). 
\end{split}
\end{equation*}
Therefore, 
\begin{equation}\label{eq:dist}
P\left(\frac{R_t-t}{\sqrt{t}}\geq x\right)-\Phi(x)
=\frac{1}{\sqrt{2\pi}}\left\{\left(I_1(t,x)
-\int_x^{\infty}e^{-u^2/2}\,{\rm d}u\right)+I_2(t,x)\right\}.
\end{equation}

We have 
\begin{equation}\label{eq:dist-0}
\begin{split}
I_1(t,x)-\int_{x\vee(-\sqrt{t})}^{\infty}e^{-u^2/2}\,{\rm d}u
&=\int_{x\vee (-\sqrt{t})}^{\infty}e^{-v^2/2}(1-e^{-2(t+v\sqrt{t})})\,{\rm d}v
-\int_{x\vee(-\sqrt{t})}^{\infty}e^{-u^2/2}\,{\rm d}u \\
&=-\int_{x\vee (-\sqrt{t})}^{\infty}e^{-v^2/2}e^{-2(t+v\sqrt{t})}\,{\rm d}v.
\end{split}
\end{equation}
Then by the change of variables formula again with $u=v+2\sqrt{t}$, 
there exists a constant $c_1>0$ such that for any $x\in {\mathbb R}$ and $t\geq 1$, 
\begin{equation}\label{eq:bound}
\int_{x\vee (-\sqrt{t})}^{\infty}e^{-v^2/2}e^{-2(t+v\sqrt{t})}\,{\rm d}v
=\int_{(x+2\sqrt{t})\vee \sqrt{t}}^{\infty}e^{-u^2/2}\,{\rm d}u
\le \int_{\sqrt{t}}^{\infty}e^{-u^2/2}\,{\rm d}u 
\le \frac{c_1e^{-t/2}}{\sqrt{t}}.
\end{equation}
Hence for any $x\in {\mathbb R}$ and $t\ge 1$,
$$
0\ge I_1(t,x)-\int_{x\vee(-\sqrt{t})}^{\infty}e^{-u^2/2}\,{\rm d}u
\ge -\frac{c_1e^{-t/2}}{\sqrt{t}}.
$$
We also note that for any $x\in {\mathbb R}$ and $t\geq 1$, 
$$\int_x^{x\vee (-\sqrt{t})}e^{-u^2/2}\,{\rm d}u
\leq \int_{-\infty}^{-\sqrt{t}}e^{-u^2/2}\,{\rm d}u
\lesssim \frac{e^{-t/2}}{\sqrt{t}}.
$$
By combining two inequalities above, 
there exists a constant $c_2>0$ such that for any $x\in {\mathbb R}$ and $t\ge 1$, 
$$
I_1(t,x)-\int_x^{\infty}e^{-u^2/2}\,{\rm d}u
=\left(I_1(t,x)-\int_{x\vee (-\sqrt{t})}^{\infty}e^{-u^2/2}\,{\rm d}u\right)
-\int_x^{x\vee (-\sqrt{t})}e^{-u^2/2}\,{\rm d}u
\ge -\frac{c_2e^{-t/2}}{\sqrt{t}}.
$$
Therefore, 
\begin{equation}\label{eq:i-bound}
\left|I_1(t,x)-\int_x^{\infty}e^{-u^2/2}\,{\rm d}u\right|
\lesssim \frac{e^{-t/2}}{\sqrt{t}}, \quad t\ge 1.
\end{equation}
Since $(e^{-v^2/2})'=-ve^{-v^2/2}$, we have by the integration by parts formula, 
\begin{equation}\label{eq:dist-2}
\begin{split}
I_2(t,x)
&=\frac{1}{\sqrt{t}}\int_{x\vee (-\sqrt{t})}^{\infty}ve^{-v^2/2}(1-e^{-2(t+v\sqrt{t})})\,{\rm d}v\\
&=\frac{e^{-(x\vee (-\sqrt{t}))^2/2}}{\sqrt{t}}(1-e^{-2(t+(x\vee (-\sqrt{t}))\sqrt{t})})
+2\int_{x\vee (-\sqrt{t})}^{\infty}e^{-v^2/2}e^{-2(t+v\sqrt{t})}\,{\rm d}v\ge 0.
\end{split}
\end{equation}
In particular, we see by \eqref{eq:bound} that 
$$
0\le I_2(t,x)
\leq \frac{1}{\sqrt{t}}\int_0^{\infty}ve^{-v^2/2}\,{\rm d}v
=\frac{1}{\sqrt{t}}.
$$
Combining this with \eqref{eq:dist} and \eqref{eq:i-bound}, 
we get \eqref{eq:berry-1} for $d=3$.

We next prove \eqref{eq:berry-2}. 
If we take $x=0$, 
then by a calculation similar to \eqref{eq:dist-0} and \eqref{eq:bound}, 
it follows that for some $c_3>0$,
$$
I_1(t,0)-\int_0^{\infty}e^{-u^2/2}\,{\rm d}u
=-\int_{2\sqrt{t}}^{\infty}e^{-u^2/2}\,{\rm d}u\ge -\frac{c_3e^{-2t}}{\sqrt{t}}.
$$
We also have for any $t\geq 1$, 
$$
I_2(t,0)
=\frac{1}{\sqrt{t}}\int_0^{\infty}ve^{-v^2/2}\left(1-e^{-2(t+v\sqrt{t})}\right)\,{\rm d}v
\geq \frac{1-e^{-2}}{\sqrt{t}}\int_0^{\infty}ve^{-v^2/2}\,{\rm d}v
=\frac{1-e^{-2}}{\sqrt{t}}.
$$
Hence by \eqref{eq:dist}, 
there exists a constant $c_4>0$ such that for any $t\geq 1$,
$$
P\left(\frac{R_t-t}{\sqrt{t}}\geq 0\right)-\Phi(0)
=\frac{1}{\sqrt{2\pi}}\left\{\left(I_1(t,0)-\int_0^{\infty}e^{-u^2/2}\,{\rm d}u\right)+I_2(t,0)\right\}
\ge \frac{c_4}{\sqrt{t}},
$$
which implies \eqref{eq:berry-2} for $d=3$.
\end{proof}

\begin{proof}[Proof of Theorem {\rm \ref{thm:berry}} for odd dimensions]
Suppose that  $d\ge 3$ is odd. 
Let 
$$
T=T(t,x)=\left(x\sqrt{t}+\frac{d-1}{2}t\right)\vee 0.
$$
Since $R_t^{(d)}\geq 0$, we have by \eqref{eq:dist-radial} and \eqref{eq:millson}, 
\begin{equation}\label{eq:d-dist}
\begin{split}
P\left(\frac{R_t^{(d)}-(d-1)t/2}{\sqrt{t}}\geq x\right)
&=\omega_d\int_T^{\infty}q_d(t,r)\sinh^{d-1}r \,{\rm d}r\\
&=-\frac{\omega_de^{-(d-2)t/2}}{2\pi}
\int_T^{\infty}\frac{\partial q_{d-2}}{\partial r}(t,r)\sinh^{d-2}r \,{\rm d}r.
\end{split}
\end{equation}
Then by the integration by parts formula and \eqref{eq:heat-bound}, 
\begin{equation*}
\begin{split}
-\int_T^{\infty}\frac{\partial q_{d-2}}{\partial r}(t,r)\sinh^{d-2}r \,{\rm d}r
&=\left[-q_{d-2}(t,r)\sinh^{d-2}r\right]_{r=T}^{r=\infty}
+\int_{T}^{\infty}q_{d-2}(t,r)\left(\frac{\partial}{\partial r}\sinh^{d-2}r\right) \,{\rm d}r\\
&=q_{d-2}(t,T)\sinh^{d-2}T
+\int_T^{\infty}q_{d-2}(t,r)\left(\frac{\partial}{\partial r}\sinh^{d-2}r\right) \,{\rm d}r.
\end{split}
\end{equation*}
In the same way, we have 
\begin{equation*}
\begin{split}
\int_T^{\infty}q_{d-2}(t,r)\left(\frac{\partial}{\partial r}\sinh^{d-2}r\right) \,{\rm d}r
&=-\frac{e^{-(d-4)t/2}}{2\pi}\int_T^{\infty}\frac{1}{\sinh r}\frac{\partial q_{d-4}}{\partial r}(t,r)
\left(\frac{\partial}{\partial r}\sinh^{d-2}r\right) \,{\rm d}r\\
&=\frac{e^{-(d-4)t/2}}{2\pi}q_{d-4}(t,T)
\left(\frac{1}{\sinh r}\frac{\partial}{\partial r}\right)\sinh^{d-2}r\Big|_{r=T}\\
&+\frac{e^{-(d-4)t/2}}{2\pi}\int_T^{\infty}q_{d-4}(t,r)
\frac{\partial}{\partial r}\left(\frac{1}{\sinh r}\frac{\partial}{\partial r}\right)\sinh^{d-2}r \,{\rm d}r.
\end{split}
\end{equation*}
Repeating this procedure, 
we see from \eqref{eq:d-dist} that for any $n\geq 2$ with $d\geq 2n$,
\begin{equation}\label{eq:d-dist-1}
\begin{split}
&P\left(\frac{R_t^{(d)}-(d-1)t/2}{\sqrt{t}}\geq x\right)\\
&=\omega_d\sum_{m=1}^{n-1}\prod_{j=1}^m\frac{e^{-(d-2j)t/2}}{2\pi}q_{d-2m}(t,T)
\left(\frac{1}{\sinh r}\frac{\partial}{\partial r}\right)^{m-1}\sinh^{d-2}r\Big|_{r=T}\\
&+\omega_d\prod_{j=1}^{n-1}\frac{e^{-(d-2j)t/2}}{2\pi} 
\int_T^{\infty}
\left(\left(\frac{1}{\sinh r}\frac{\partial}{\partial r}\right)^{n-1}\sinh^{d-2} r\right)q_{d-2n+2}(t,r)\sinh r\,{\rm d}r\\
&=\omega_d\sum_{m=1}^{n-1}\frac{e^{-(d-(m+1))mt/2}}{(2\pi)^m} q_{d-2m}(t,T)
\left(\frac{1}{\sinh r}\frac{\partial}{\partial r}\right)^{m-1}\sinh^{d-2}r\Big|_{r=T}\\
&+\omega_d\frac{e^{-(d-n)(n-1)t/2}}{(2\pi)^{n-1}}
\int_T^{\infty}
\left(\left(\frac{1}{\sinh r}\frac{\partial}{\partial r}\right)^{n-1}\sinh^{d-2} r\right)q_{d-2n+2}(t,r)\sinh r\,{\rm d}r.
\end{split}
\end{equation}

Let us calculate the first term of the last expression  of \eqref{eq:d-dist-1}. 
We now assume that  $x\geq -(d-1)\sqrt{t}/2$ and so $T\ge 0$. 
Then by \eqref{eq:ind-1} and \eqref{eq:ind-2}, 
for any $m$ \ ($1\le m\le n-1$), there exists a constant $c_1>0$ such that 
$$
\left(\frac{1}{\sinh r}\frac{\partial}{\partial r}\right)^{m-1}\sinh^{d-2}r\Big|_{r=T}
\leq c_1e^{(d-m-1)T}.
$$
We also see from  \eqref{eq:heat-bound} that 
\begin{equation*}
\begin{split}
&q_{d-2m}(t,T)\\
&\asymp \frac{1}{t^{(d-2m)/2}}
\exp\left(-\frac{(d-2m-1)^2}{8}t-\frac{(d-2m-1)T}{2}-\frac{T^2}{2t}\right)
(1+T+t)^{(d-2m-3)/2}(1+T).
\end{split}
\end{equation*}
Since there exists a constant $c_2>0$ 
such that for any $t\geq 1$ and $x\geq -(d-1)\sqrt{t}/2$,
\begin{equation*}
\begin{split}
(1+T+t)^{(d-2m-3)/2}(1+T)
&\le (1+T+t)^{(d-2m-1)/2}=(1+x\sqrt{t}+(d-1)t/2+t)^{(d-2m-1)/2}\\
&\le c_2(1+(|x|\sqrt{t})^{(d-2m-1)/2}+t^{(d-2m-1)/2}),
\end{split}
\end{equation*}
we can take a constant $c_3>0$ such that for any $t\geq 1$ and $x\geq -(d-1)\sqrt{t}/2$,
$$
\frac{e^{-x^2/2}}{t^{(d-2m)/2}}(1+T+t)^{(d-2m-3)/2}(1+T)
\leq c_2\left(\frac{e^{-x^2/2}}{t^{d-2m}}+\frac{e^{-x^2/2}|x|^{(d-2m-1)/2}}{t^{(d-2m)/4+1/4}}+\frac{1}{\sqrt{t}}\right)
\leq \frac{c_3}{\sqrt{t}}.
$$
At the last inequality, we used the fact that 
$\sup_{x\in {\mathbb R}}e^{-x^2/2}|x|^{(d-2m-1)/2}<\infty$. 
By noting that 
$$
\exp\left(-\frac{(d-2m-1)^2}{8}t-\frac{(d-2m-1)T}{2}-\frac{T^2}{2t}\right)e^{(d-m-1)T}
=e^{-x^2/2}e^{(d-(m+1))mt/2},
$$
there exist  positive constants $c_4$ and $c_5$ such that 
for any $t\geq 1$ and $x\geq -(d-1)\sqrt{t}/2$,
\begin{equation}\label{eq:bound-large}
\begin{split}
&\frac{e^{-(d-(m+1))mt/2}}{(2\pi)^m} q_{d-2m}(t,T)
\left(\frac{1}{\sinh r}\frac{\partial}{\partial r}\right)^{m-1}\sinh^{d-2}r\Big|_{r=T}\\
&\le \frac{c_4e^{-x^2/2}}{t^{(d-2m)/2}}(1+T+t)^{(d-2m-3)/2}(1+T)
\le \frac{c_5}{\sqrt{t}}.
\end{split}
\end{equation}

On the other hand, if $x\leq -(d-1)\sqrt{t}/2$, 
then $T=0$ so that by \eqref{eq:ind-1} and \eqref{eq:ind-2},
$$
\left(\frac{1}{\sinh r}\frac{\partial}{\partial r}\right)^{m-1}\sinh^{d-2}r\Big|_{r=0}=0.
$$
Combining this with \eqref{eq:bound-large}, 
we have for any $t\geq 1$,
$$
\sup_{x\in {\mathbb R}} \omega_d\sum_{m=1}^{n-1}
\frac{e^{-(d-(m+1))mt/2}}{(2\pi)^m}
q_{d-2m}(t,T)
\left(\frac{1}{\sinh r}\frac{\partial}{\partial r}\right)^{m-1}\sinh^{d-2}r\Big|_{r=T}
\lesssim \frac{1}{\sqrt{t}}.
$$
In particular, if we take $d=2n+1$, then 
\begin{equation}\label{eq:bound-all}
\sup_{x\in {\mathbb R}} \omega_d\sum_{m=1}^{n-1}
\frac{e^{-(2n-m)mt/2}}{(2\pi)^m}
q_{2n+1-2m}(t,T)
\left(\frac{1}{\sinh r}\frac{\partial}{\partial r}\right)^{m-1}\sinh^{2n-1}r\Big|_{r=T}
\lesssim \frac{1}{\sqrt{t}}.
\end{equation}

We turn to the second term of the last expression  of \eqref{eq:d-dist-1}. 
Let $d=2n+1$ so that $T=(x\sqrt{t}+nt)\vee 0$. 
By \eqref{eq:d=3} and \eqref{eq:deri-0}, we have 
\begin{equation*}
\begin{split}
\left(\left(\frac{1}{\sinh r}\frac{\partial}{\partial r}\right)^{n-1}\sinh^{2n-1}r\right)q_3(t,r)\sinh r 
&=\frac{(2n-1)!!}{n}\sinh(nr)\frac{e^{-t/2}}{(2\pi t)^{3/2}}\frac{r}{\sinh r}
e^{-r^2/(2t)}\sinh r\\
&=\frac{(2n-1)!!}{n}\frac{e^{-t/2}}{(2\pi t)^{3/2}}
re^{-r^2/(2t)}\sinh(nr)
\end{split}
\end{equation*}
and thus
\begin{equation*}
\begin{split}
&\frac{e^{-((2n+1)-n)(n-1)t/2}}{(2\pi)^{n-1}}
\int_T^{\infty}
\left(\left(\frac{1}{\sinh r}\frac{\partial}{\partial r}\right)^{n-1}\sinh^{2n-1}r\right)q_3(t,r)\sinh r\,{\rm d}r\\
&=\frac{e^{-(n^2-1)t/2}}{(2\pi)^{n-1}}\frac{(2n-1)!!}{n}  \frac{e^{-t/2}}{(2\pi t)^{3/2}}
\int_T^{\infty}re^{-r^2/(2t)}\sinh(nr)\,{\rm d}r\\
&=\frac{1}{(2\pi)^{n-1}(2\pi t)^{3/2}}\frac{(2n-1)!!}{n}  
\int_T^{\infty}re^{-n^2t/2}e^{-r^2/(2t)}\sinh(nr)\,{\rm d}r.
\end{split}
\end{equation*}
Since
$$
\omega_{2n+1}=\frac{(2\pi)^{n-1} 4\pi}{(2n-1)!!}
=\frac{(2\pi)^{n-1}\omega_3}{(2n-1)!!},
$$
we get by \eqref{eq:d=3},
\begin{equation*}
\begin{split}
&
\omega_{2n+1}\frac{e^{-((2n+1)-n)(n-1)t/2}}{(2\pi)^{n-1}}
\int_T^{\infty}
\left(\left(\frac{1}{\sinh r}\frac{\partial}{\partial r}\right)^{n-1} \sinh^{2n-1}r\right)q_3(t,r)\sinh r\,{\rm d}r\\
&=\frac{\omega_3}{n(2\pi t)^{3/2}}
\int_T^{\infty}re^{-n^2t/2}e^{-r^2/(2t)}\sinh(nr)\,{\rm d}r\\
&=\frac{\omega_3}{(2\pi n^2t)^{3/2}}
\int_{(x\sqrt{n^2 t}+n^2 t)\vee 0}^{\infty}
\frac{u}{\sinh u}e^{-n^2t/2}e^{-u^2/(2n^2t)}\sinh^2 u\,{\rm d}u\\
&=\frac{\omega_3}{(2\pi n^2t)^{3/2}}\int_{(x\sqrt{n^2 t}+n^2 t)\vee 0}^{\infty}
q_3(n^2t,u)\sinh^2 u\,{\rm d}u
=P\left(\frac{R_{n^2 t}^{(3)}-n^2t}{\sqrt{n^2t}}\geq x\right).
\end{split}
\end{equation*}
At the second equality above, we used the change of variable formula with $u=nr$. 
Hence by \eqref{eq:d-dist-1},
\begin{equation}\label{eq:distance}
\begin{split}
&P\left(\frac{R_t^{(2n+1)}-nt}{\sqrt{t}}\geq x\right)-\Phi(x)\\
&=\omega_{2n+1}\sum_{m=1}^{n-1} \frac{e^{-(2n-m)mt/2}}{(2\pi)^m}
q_{2n+1-2m}(t,T)
\left(\frac{1}{\sinh r}\frac{\partial}{\partial r}\right)^{m-1}\sinh^{2n-1}r\Big|_{r=T}\\
&+P\left(\frac{R_{n^2 t}^{(3)}-n^2t}{\sqrt{n^2t}}\geq x\right)
-\Phi(x). 
\end{split}
\end{equation}
Since Theorem \ref{thm:berry} is already proved for $d=3$, 
we see by \eqref{eq:bound-all} that 
$$
\sup_{x\in {\mathbb R}}\left|P\left(\frac{R_t^{(2n+1)}-nt}{\sqrt{t}}\geq x\right)-\Phi(x)\right|
\lesssim \frac{1}{\sqrt{t}}, 
\quad t\geq 1.
$$
The proof of \eqref{eq:berry-1} is complete for general odd dimensions.

We next prove \eqref{eq:berry-2} for general odd dimensions. 
It follows by \eqref{eq:ind-1} and \eqref{eq:ind-2}, that, 
for any $m$ \ ($1\le m\le n-1$), there exists a constant $c_6>0$ such that 
$$
\left(\frac{1}{\sinh r}\frac{\partial}{\partial r}\right)^{m-1}\sinh^{2n-1}r\Big|_{r=nt}
\ge c_6e^{(2n-m)nt}.
$$
Hence as in the proof of \eqref{eq:bound-large}, 
we have for some $c_7>0$, 
\begin{equation}\label{eq:bound-lower}
\sum_{m=1}^{n-1}
\frac{e^{-(2n-m)mt/2}}{(2\pi)^m}
q_{2n+1-m}(t,nt)
\left(\frac{1}{\sinh r}\frac{\partial}{\partial r}\right)^{m-1}\sinh^{2n-1}r\Big|_{r=nt}
\ge\frac{c_7}{\sqrt{t}}, \quad t\ge 1.
\end{equation}
Since \eqref{eq:berry-2} is already proved for $d=3$, 
we see from \eqref{eq:distance} (with $x=0$) and \eqref{eq:bound-lower} that 
there exists a constant $c_8>0$ such that 
$$
P\left(\frac{R_t^{(2n+1)}-nt}{\sqrt{t}}\geq 0\right)
-\Phi(0)
\ge \frac{c_8}{\sqrt{t}}, 
\quad t\geq 1.
$$
The proof of \eqref{eq:berry-2} is complete for general odd dimensions.
\end{proof}

\subsection{Even dimensions}
In this subsection, we establish Theorem \ref{thm:berry} for even dimensions.
Suppose that $d=2n$ for some $n\ge 1$.  
Then $T=(x\sqrt{t}+(n-1/2)t)\vee 0$. By \eqref{eq:d-dist-1},
\begin{equation}\label{eq:i+ii}
\begin{split}
&P\left(\frac{R_t^{(2n)}-(n-1/2)t}{\sqrt{t}}\ge x\right)\\
&=\omega_{2n}\sum_{m=1}^{n-1}\frac{e^{-m(n-(m+1)/2)t}}{(2\pi)^m}q_{2n-2m}(t,T)
\left(\frac{1}{\sinh r}\frac{\partial}{\partial r}\right)^{m-1}\sinh^{2n-2}r\Big|_{r=T}\\
&+\omega_{2n}\frac{e^{-n(n-1)t/2}}{(2\pi)^{n-1}} 
\int_T^{\infty}
\left(\left(\frac{1}{\sinh r}\frac{\partial}{\partial r}\right)^{n-1}\sinh^{2n-2} r\right)q_2(t,r)\sinh r\,{\rm d}r\\
&=J_1(t,x)+J_2(t,x),
\end{split}
\end{equation}
where we make the convention $\sum_{k=1}^0=0$ so that $J_1(t,x)=0$ for $n=1$.
Therefore,
\begin{equation}\label{eq:d-dist-even}
P\left(\frac{R_t^{(2n)}-(n-1/2)t}{\sqrt{t}}\ge x\right)-\Phi(x)
=J_1(t,x)+\left(J_2(t,x)-\Phi(x)\right).
\end{equation}
We now estimate the right hand side above.

\begin{lem}\label{lem:i-bound}
There exists a constant $c>0$ such that 
for any $t\ge 1$, 
$$
\sup_{x\in {\mathbb R}}J_1(t,x)\le \frac{c}{\sqrt{t}}.
$$
\end{lem}

We omit the proof of Lemma \ref{lem:i-bound} 
because the argument for \eqref{eq:bound-all} still works.

\begin{lem}\label{lem:ii-bound}
There exists a constant $c>0$ such that for any $t\geq 1$, 
$$
\sup_{x\in {\mathbb R}}\left|J_2(t,x)-\Phi(x)\right|
\le \frac{c}{\sqrt{t}}.
$$
\end{lem}

\begin{proof}
Let 
$$
a_n(t)=\omega_{2n}\frac{e^{-n(n-1)t/2}}{(2\pi)^{n-1}}=\frac{\pi e^{-n(n-1)t/2}}{2^{n-2}(n-1)!}
$$
and
$$
K(t,x)
=\int_T^{\infty}
\left(\left(\frac{1}{\sinh r}\frac{\partial}{\partial r}\right)^{n-1}
\sinh^{2n-2} r\right)q_2(t,r)\sinh r\,{\rm d}r
$$
so that 
\begin{equation}\label{eq:ii-divide-0}
J_2(t,x)=a_n(t)K(t,x).
\end{equation}
Then by \eqref{eq:d=2} and the Fubini theorem,
\begin{equation}\label{eq:d-dist-even-1}
\begin{split}
&K(t,x)\\
&=\frac{2^{1/2}e^{-t/8}}{(2\pi t)^{3/2}}
\int_T^{\infty}s e^{-s^2/(2t)}
\left\{\int_T^s \left(\left(\frac{1}{\sinh r}\frac{\partial}{\partial r}\right)^{n-1}\sinh^{2n-2} r\right)
\frac{\sinh r}{(\cosh s-\cosh r)^{1/2}}\,{\rm d}r\right\}\,{\rm d}s.
\end{split}
\end{equation}
By the integration by parts formula,
\begin{equation*}
\begin{split}
&\int_T^s \left(\left(\frac{1}{\sinh r}\frac{\partial}{\partial r}\right)^{n-1}\sinh^{2n-2} r\right)
\frac{\sinh r}{(\cosh s-\cosh r)^{1/2}}\,{\rm d}r\\
&=\left[-2\left(\left(\frac{1}{\sinh r}\frac{\partial}{\partial r}\right)^{n-1}\sinh^{2n-2} r\right)
(\cosh s-\cosh r)^{1/2}\right]_{r=T}^{r=s}\\
&+2\int_T^s \left(\left(\frac{1}{\sinh r}\frac{\partial}{\partial r}\right)^{n}
\sinh^{2n-2} r\right)(\cosh s-\cosh r)^{1/2}\sinh r \,{\rm d}r\\
&=2\left(\left(\frac{1}{\sinh r}\frac{\partial}{\partial r}\right)^{n-1}\sinh^{2n-2} r\right)\Bigg|_{r=T}
(\cosh s-\cosh T)^{1/2}\\
&+2\int_T^s \left(\left(\frac{1}{\sinh r}\frac{\partial}{\partial r}\right)^{n}
\sinh^{2n-2} r\right)(\cosh s-\cosh r)^{1/2}\sinh r \,{\rm d}r.
\end{split}
\end{equation*}
Inductively, we get 
\begin{equation}\label{eq:ind-3}
\begin{split}
&\int_T^s \left(\left(\frac{1}{\sinh r}\frac{\partial}{\partial r}\right)^{n-1}\sinh^{2n-2} r\right)
\frac{\sinh r}{(\cosh s-\cosh r)^{1/2}}\,{\rm d}r\\
&=\sum_{k=1}^{n-1}\frac{2^k}{(2k-1)!!}
\left(\left(\frac{1}{\sinh r}\frac{\partial}{\partial r}\right)^{n-2+k}\right)\sinh^{2n-2}r\,\Big|_{r=T}
(\cosh s-\cosh T)^{k-1/2}\\
&+\frac{2^{n-1}}{(2n-3)!!}\int_T^s \left(\left(\frac{1}{\sinh r}\frac{\partial}{\partial r}\right)^{2n-2}\sinh^{2n-2}r\right)
(\cosh s-\cosh r)^{n-3/2}\sinh r\,{\rm d}r.
\end{split}
\end{equation}
Since it follows by \eqref{eq:ind-1} that 
$$
\left(\frac{1}{\sinh r}\frac{\partial}{\partial r}\right)^{2n-2}\sinh^{2n-2}r
=(2n-2)!,
$$
we have 
\begin{equation*}
\begin{split}
&\int_T^s \left(\left(\frac{1}{\sinh r}\frac{\partial}{\partial r}\right)^{2n-2}\sinh^{2n-2}r\right)
(\cosh s-\cosh r)^{n-3/2}\sinh r\,{\rm d}r\\
&=(2n-2)!
\int_T^s (\cosh s-\cosh r)^{n-3/2}\sinh r\,{\rm d}r
=\frac{2(2n-2)!}{2n-1}(\cosh s-\cosh T)^{n-1/2}
\end{split}
\end{equation*}
and thus 
\begin{equation*}
\begin{split}
&\frac{2^{n-1}}{(2n-3)!!}\int_T^s \left(\left(\frac{1}{\sinh r}\frac{\partial}{\partial r}\right)^{2n-2}\sinh^{2n-2}r\right)
(\cosh s-\cosh r)^{n-3/2}\sinh r\,{\rm d}r\\
&=\frac{2^n(2n-2)!}{(2n-1)!!}(\cosh s-\cosh T)^{n-1/2}
=\frac{2^{2n-1}(n-1)!}{2n-1}(\cosh s-\cosh T)^{n-1/2}.
\end{split}
\end{equation*}
Combining this with \eqref{eq:d-dist-even-1} and \eqref{eq:ind-3}, 
we obtain 
\begin{equation}\label{eq:d-dist-even-2}
\begin{split}
K(t,x)
&=\frac{2^{1/2}e^{-t/8}}{(2\pi t)^{3/2}}
\sum_{k=1}^{n-1}\frac{2^k}{(2k-1)!!}
\left(\left(\frac{1}{\sinh r}\frac{\partial}{\partial r}\right)^{n-2+k}\right)\sinh^{2n-2}r\,\Big|_{r=T}\\
&\times \int_T^{\infty}s e^{-s^2/(2t)}
(\cosh s-\cosh T)^{k-1/2}\,{\rm d}s\\
&+\frac{2^{2n-1}(n-1)!}{2n-1}\frac{2^{1/2}e^{-t/8}}{(2\pi t)^{3/2}}
\int_T^{\infty}s e^{-s^2/(2t)}
(\cosh s-\cosh T)^{n-1/2}\,{\rm d}s\\
&=K_1(t,x)+K_2(t,x),
\end{split}
\end{equation}
whence by \eqref{eq:ii-divide-0},
\begin{equation}\label{eq:ii-divide}
J_2(t,x)=a_n(t)K_1(t,x)+a_n(t)K_2(t,x).
\end{equation}

Let us first estimate $a_n(t)K_1(t,x)$. 
By \eqref{eq:ind-1} and \eqref{eq:ind-2}, 
there exists a constant $c_1>0$, 
which is independent of $t$, $T$ and $k \ (1\le k\le n-1)$, 
such that  
\begin{equation}\label{eq:sinh-deri}
0\le 
\left(\left(\frac{1}{\sinh r}\frac{\partial}{\partial r}\right)^{n-2+k}\right)\sinh^{2n-2}r\,\Big|_{r=T}
\le c_1 e^{(n-k)T}=c_1 e^{(n-k)((x\sqrt{t}+(n-1/2)t)\vee 0)}.
\end{equation}
By the change of variables formula with $s=u\sqrt{t}+(n-1/2)t$, 
we also obtain
\begin{equation*}
\begin{split}
&\frac{2^{1/2}e^{-t/8}}{(2\pi t)^{3/2}}\int_T^{\infty}s e^{-s^2/(2t)}(\cosh s-\cosh T)^{k-1/2}\,{\rm d}s\\
&=\frac{2^{1/2}e^{-t/8}}{(2\pi t)^{3/2}}
\sqrt{t}\int_{x\vee (-(n-1/2)\sqrt{t})}^{\infty}
\left(u\sqrt{t}+\left(n-\frac{1}{2}\right)t\right) e^{-(u+(n-1/2)\sqrt{t})^2/2}\\
&\times \left(\cosh\left(u\sqrt{t}+\left(n-\frac{1}{2}\right)t\right)
-\cosh \left(\left(x\sqrt{t}+\left(n-\frac{1}{2}\right)t\right)\vee 0\right)\right)^{k-1/2}\,{\rm d}u.
\end{split}
\end{equation*}
Then for any $k\ge 1$ and $u\ge x\vee (-(n-1/2)\sqrt{t})$, we have 
\begin{equation*}
\begin{split}
0
&\le \left(\cosh \left(u\sqrt{t}+\left(n-\frac{1}{2}\right)t\right)
-\cosh \left(\left(x\sqrt{t}+\left(n-\frac{1}{2}\right)t\right)\vee 0\right)\right)^{k-1/2}\\
&\le \cosh^{k-1/2} \left(u\sqrt{t}+\left(n-\frac{1}{2}\right)t\right)
\le e^{(u\sqrt{t}+(n-1/2)t)(k-1/2)}
\end{split}
\end{equation*}
so that for any $t\ge 1$,
\begin{equation*}
\begin{split}
0&\le\frac{2^{1/2}e^{-t/8}}{(2\pi t)^{3/2}}
\int_T^{\infty}s e^{-s^2/(2t)}(\cosh s-\cosh T)^{k-1/2}\,{\rm d}s \\
&\le c_2e^{-t/8}
\int_{x\vee (-(n-1/2)\sqrt{t})}^{\infty}
\left(\frac{|u|}{\sqrt{t}}+1\right) e^{-(u+(n-1/2)\sqrt{t})^2/2}  e^{(u\sqrt{t}+(n-1/2)t)(k-1/2)}\,{\rm d}u\\
&=c_2e^{-(n^2-2nk+k)t/2}
\int_{x\vee (-(n-1/2)\sqrt{t})}^{\infty}
\left(\frac{|u|}{\sqrt{t}}+1\right) e^{-u^2/2} e^{-(n-k)\sqrt{t}u}\,{\rm d}u\\
&\le c_3e^{-(n^2-2nk+k)t/2}
\int_{x\vee (-(n-1/2)\sqrt{t})}^{\infty} e^{-(n-k)\sqrt{t}u}\,{\rm d}u
=c_3e^{-(n^2-2nk+k)t/2}\frac{e^{-(n-k)(x\sqrt{t}\vee (-(n-1/2)t))}}{(n-k)\sqrt{t}}.
\end{split}
\end{equation*}
Combining this with \eqref{eq:sinh-deri}, we get 
\begin{equation*}
\begin{split}
0&\le
\left(\left(\frac{1}{\sinh r}\frac{\partial}{\partial r}\right)^{n-2+k}\right)\sinh^{2n-2}r\,\Big|_{r=T}
\frac{2^{1/2}e^{-t/8}}{(2\pi t)^{3/2}}
\int_T^{\infty}s e^{-s^2/(2t)}(\cosh s-\cosh T)^{k-1/2}\,{\rm d}s\\
&\le c_4e^{(n-k)((x\sqrt{t}+(n-1/2)t)\vee 0)}e^{-(n^2-2nk+k)t/2}
\frac{e^{-(n-k)(x\sqrt{t}\vee (-(n-1/2)t))}}{(n-k)\sqrt{t}}
=\frac{c_4e^{n(n-1)t/2}}{\sqrt{t}}.
\end{split}
\end{equation*}
This implies that 
$$
0\le K_1(t,x)\le \frac{c_5e^{n(n-1)t/2}}{\sqrt{t}}
$$
and thus
\begin{equation}\label{eq:k1-bound}
a_n(t)K_1(t,x)=\frac{\pi e^{-n(n-1)t/2}}{2^{n-2}(n-1)!}K_1(t,x)
\le \frac{c_6}{\sqrt{t}}.
\end{equation}

Let us next estimate $a_n(t)K_2(t,x)$. 
By definition, 
\begin{equation*}
\begin{split}
a_n(t)K_2(t,x)
&=\frac{2^n}{(2n-1)\sqrt{\pi}}\frac{e^{-(n-1/2)^2t/2}}{t^{3/2}}\int_T^{\infty}s e^{-s^2/(2t)}(\cosh s-\cosh T)^{n-1/2}\,{\rm d}s\\
&=b_n(t)\int_T^{\infty}s e^{-s^2/(2t)}(\cosh s-\cosh T)^{n-1/2}\,{\rm d}s.
\end{split}
\end{equation*}
We first assume that $x\ge -(n-1/2)\sqrt{t}$. 
Then by the change of variables formula with $s=u\sqrt{t}+(n-1/2)t$,  
\begin{equation*}
\begin{split}
&\int_T^{\infty}s e^{-s^2/(2t)}(\cosh s-\cosh T)^{n-1/2}\,{\rm d}s\\
&=te^{-(n-1/2)^2t/2}\int_x^{\infty}
ue^{-u^2/2}e^{-(n-1/2)u\sqrt{t}}\left(\cosh \left(u\sqrt{t}+\left(n-\frac{1}{2}\right)t\right)-\cosh T\right)^{n-1/2}\,{\rm d}u\\
&+\left(n-\frac{1}{2}\right)t\sqrt{t}e^{-(n-1/2)^2t/2}\\
&\times 
\int_x^{\infty}
e^{-u^2/2}e^{-(n-1/2)u\sqrt{t}}\left(\cosh \left(u\sqrt{t}+\left(n-\frac{1}{2}\right)t\right)-\cosh T\right)^{n-1/2}\,{\rm d}u\\
&=L_1(t,x)+L_2(t,x).
\end{split}
\end{equation*}
Therefore, 
\begin{equation}\label{eq:k2-bound-1}
a_n(t)K_2(t,x)=b_n(t)L_1(t,x)+b_n(t)L_2(t,x).
\end{equation}
Since we have for any $u\ge x$,
\begin{equation*}
\begin{split}
e^{-(n-1/2)u\sqrt{t}}\left(\cosh \left(u\sqrt{t}+\left(n-\frac{1}{2}\right)t\right)-\cosh T\right)^{n-1/2}
&\le e^{-(n-1/2)u\sqrt{t}}\left(e^{u\sqrt{t}+(n-1/2)t}\right)^{n-1/2}\\
&=e^{(n-1/2)^2t},
\end{split}
\end{equation*}
it follows that 
\begin{equation*}
\begin{split}
|L_1(t,x)|
&\le te^{(n-1/2)^2t/2}\int_x^{\infty}|u|e^{-u^2/2}\,{\rm d}u
\le te^{(n-1/2)^2t/2}\int_{-\infty}^{\infty}|u|e^{-u^2/2}\,{\rm d}u=2te^{(n-1/2)^2t/2},
\end{split}
\end{equation*}
which yields
\begin{equation}\label{eq:k2-bound-2}
|b_n(t)L_1(t,x)|\le \frac{c_7}{\sqrt{t}}.
\end{equation}

Since 
\begin{equation*}
\begin{split}
&\left(\cosh \left(u\sqrt{t}+\left(n-\frac{1}{2}\right)t\right)-\cosh T\right)^{n-1/2}\\
&=\cosh^{n-1/2}\left(u\sqrt{t}+\left(n-\frac{1}{2}\right)t\right)
\left(1-\frac{\cosh T}{\cosh \left(u\sqrt{t}+(n-1/2)t\right)}\right)^{n-1/2}\\
&=\frac{1}{2^{n-1/2}}e^{(n-1/2)^2t}e^{(n-1/2)u\sqrt{t}}
\left(1+e^{-2(u\sqrt{t}+(n-1/2)t)}\right)^{n-1/2}\\
&\times\left(1-\frac{\cosh T}{\cosh \left(u\sqrt{t}+(n-1/2)t\right)}\right)^{n-1/2},
\end{split}
\end{equation*}
we obtain
\begin{equation*}
\begin{split}
L_2(t,x)
&=\frac{2n-1}{2^{n+1/2}}t\sqrt{t}e^{(n-1/2)^2t/2}\\
&\times \int_x^{\infty}
e^{-u^2/2}\left(1+e^{-2(u\sqrt{t}+(n-1/2)t)}\right)^{n-1/2}
\left(1-\frac{\cosh T}{\cosh \left(u\sqrt{t}+(n-1/2)t\right)}\right)^{n-1/2}\,{\rm d}u.
\end{split}
\end{equation*}
Noting that 
$$
b_n(t)\frac{2n-1}{2^{n+1/2}}t\sqrt{t}e^{(n-1/2)^2t/2}=\frac{1}{\sqrt{2\pi}},
$$
we further have 
\begin{equation*}
\begin{split}
&b_n(t)L_2(t,x)\\
&=\frac{1}{\sqrt{2\pi}}
\int_x^{\infty}
e^{-u^2/2}\left(1+e^{-2(u\sqrt{t}+(n-1/2)t)}\right)^{n-1/2}
\left(1-\frac{\cosh T}{\cosh \left(u\sqrt{t}+(n-1/2)t\right)}\right)^{n-1/2}\,{\rm d}u,
\end{split}
\end{equation*}
which yields
\begin{equation}\label{eq:k2-bound-3}
\begin{split}
&b_n(t)L_2(t,x)-\Phi(x)
=\frac{1}{\sqrt{2\pi}}
\int_x^{\infty}
e^{-u^2/2}\left\{\left(1+e^{-2(u\sqrt{t}+(n-1/2)t)}\right)^{n-1/2}-1\right\}\,{\rm d}u\\
&+\frac{1}{\sqrt{2\pi}}\\
&\times \int_x^{\infty}
e^{-u^2/2}\left(1+e^{-2(u\sqrt{t}+(n-1/2)t)}\right)^{n-1/2}
\left\{\left(1-\frac{\cosh T}{\cosh \left(u\sqrt{t}+(n-1/2)t\right)}\right)^{n-1/2}-1\right\}\,{\rm d}u\\
&=M_1(t,x)+M_2(t,x).
\end{split}
\end{equation}

Note that for any $p\ge 0$,  
$$
(1+t)^p-1\asymp  t, \quad 0\le t\le 1.
$$
Then for any $t\ge 1$ and $u\ge -(n-1/2)\sqrt{t}$, 
$$
\left(1+e^{-2(u\sqrt{t}+(n-1/2)t)}\right)^{n-1/2}-1\asymp e^{-2(u\sqrt{t}+(n-1/2)t)}.
$$
Hence for any $t\ge 1$ and $x\ge -(n-1/2)\sqrt{t}$, 
$$
M_1(t,x)\asymp 
\int_x^{\infty}e^{-u^2/2}e^{-2(u\sqrt{t}+(n-1/2)t)}\,{\rm d}u
=e^{-(2n-3)t}\int_x^{\infty}e^{-(u+2\sqrt{t})^2/2}\,{\rm d}u.
$$
Moreover,  by the change of variables formula with $v=u+2\sqrt{t}$, 
we have for any $t\ge 1$,
$$
\int_x^{\infty}e^{-(u+2\sqrt{t})^2/2}\,{\rm d}u
=\int_{x+2\sqrt{t}}^{\infty}e^{-v^2/2}\,{\rm d}v
\le \int_{-(n-5/2)t}^{\infty}e^{-v^2/2}\,{\rm d}v
\lesssim  
\begin{cases}
e^{-9t/8}/\sqrt{t} & (n=1),\\
1 & (n\ge 2),
\end{cases}
$$
which implies that for any $t\ge 1$, 
\begin{equation}\label{eq:m1-bound}
\sup_{x\ge -(n-1/2)t}M_1(t,x)
\lesssim 
\begin{cases}
e^{-t/8}/\sqrt{t} & (n=1),\\
e^{-(2n-3)t} & (n\ge 2).
\end{cases}
\end{equation}
In particular, the right hand side above decays exponentially for any $n\ge 1$.

We also note that for any $p\ge 0$, there exists a constant $c_8>0$ such that 
\begin{equation}\label{eq:p-ineq}
0\le 1-(1-t)^p\le c_8t, \quad 0\le t\le 1.
\end{equation}
Therefore, for any $u\ge x$,
\begin{equation*}
\begin{split}
0&\le 1-\left(1-\frac{\cosh T}{\cosh \left(u\sqrt{t}+(n-1/2)t\right)}\right)^{n-1/2}
\le \frac{c_8\cosh T}{\cosh \left(u\sqrt{t}+(n-1/2)t\right)}\\
&\le 2c_8 e^Te^{-(u\sqrt{t}+(n-1/2)t)}=2c_8e^{-(u-x)\sqrt{t}},
\end{split}
\end{equation*}
which implies that 
\begin{equation*}
\begin{split}
|M_2(t,x)|
=-M_2(t,x)
&\le c_9\int_x^{\infty}e^{-u^2/2}e^{-(u-x)\sqrt{t}}\,{\rm d}u
=c_9\int_0^{\infty}e^{-(v+x)^2/2}e^{-v\sqrt{t}}\,{\rm d}v\\
&\le c_9\int_0^{\infty}e^{-v\sqrt{t}}\,{\rm d}v=\frac{c_9}{\sqrt{t}}.
\end{split}
\end{equation*}
Combining this with \eqref{eq:k2-bound-3} and \eqref{eq:m1-bound}, 
we obtain for any $t\ge 1$, 
$$
\sup_{x\ge -(n-1/2)t}\left|b_n(t)L_2(t,x)-\Phi(x)\right|
\lesssim \frac{1}{\sqrt{t}}.
$$
By \eqref{eq:k2-bound-1} and \eqref{eq:k2-bound-2}, 
we further get for any $t\ge 1$,
$$
\sup_{x\ge -(n-1/2)t}\left|a_n(t)K_2(t,x)-\Phi(x)\right|
\lesssim \frac{1}{\sqrt{t}}.
$$
Hence by \eqref{eq:ii-divide} and \eqref{eq:k1-bound},
$$
\sup_{x\ge -(n-1/2)t}\left|J_2(t,x)-\Phi(x)\right|
\lesssim \frac{1}{\sqrt{t}}, \quad t\ge 1.
$$

We next assume that $x\le -(n-1/2)\sqrt{t}$ and thus $T=0$. 
By the change of variables formula with $s=u\sqrt{t}$, we have 
\begin{equation*}
\begin{split}
0\le \int_0^{\infty}se^{-s^2/(2t)}(\cosh s-1)^{n-1/2}\,{\rm d}s
&=t\int_0^{\infty}ue^{-u^2/2}(\cosh(u\sqrt{t})-1)^{n-1/2}\,{\rm d}u\\
&=2^{n-1/2}t\int_0^{\infty}ue^{-u^2/2}\sinh^{2n-1}(u\sqrt{t}/2)\,{\rm d}u.
\end{split}
\end{equation*}
At the last equality, we used the formula $\cosh(u\sqrt{t})-1=2\sinh^2(u\sqrt{t}/2)$.
Since $ue^{-u^2/2}=-(e^{-u^2/2})'$, we have  by the integration by parts formula, 
$$
\int_0^{\infty}ue^{-u^2/2}\sinh^{2n-1}(u\sqrt{t}/2)\,{\rm d}u
=\left(n-\frac{1}{2}\right)\sqrt{t}\int_0^{\infty}e^{-u^2/2}\sinh^{2n-2}(u\sqrt{t}/2)\cosh(u\sqrt{t}/2)\,{\rm d}u,
$$
which yields
\begin{equation*}
\begin{split}
&\int_0^{\infty}se^{-s^2/(2t)}(\cosh s-1)^{n-1/2}\,{\rm d}s\\
&=2^{n-3/2}(2n-1)t\sqrt{t}\int_0^{\infty}e^{-u^2/2}\sinh^{2n-2}(u\sqrt{t}/2)\cosh(u\sqrt{t}/2)\,{\rm d}u\\
&=\frac{2n-1}{2^{n+1/2}}t\sqrt{t}
\int_0^{\infty}e^{-u^2/2}e^{(n-1/2)u\sqrt{t}}(1-e^{-u\sqrt{t}})^{2n-2}(1+e^{-u\sqrt{t}})\,{\rm d}u.
\end{split}
\end{equation*}
Hence as in \eqref{eq:k2-bound-1},
\begin{equation}\label{eq:k2-bound-11}
\begin{split}
a_n(t)K_2(t,x)
&=\frac{1}{\sqrt{2\pi}}e^{-(n-1/2)^2t/2}
\int_0^{\infty}e^{-u^2/2}e^{(n-1/2)u\sqrt{t}}(1-e^{-u\sqrt{t}})^{2n-2}(1+e^{-u\sqrt{t}})\,{\rm d}u\\
&=\frac{1}{\sqrt{2\pi}}
\int_0^{\infty}e^{-(u-(n-1/2)\sqrt{t})^2/2}(1-e^{-u\sqrt{t}})^{2n-2}\,{\rm d}u\\
&+\frac{e^{-(n-1)t}}{\sqrt{2\pi}}
\int_0^{\infty}e^{-(u-(n-3/2)\sqrt{t})^2/2}(1-e^{-u\sqrt{t}})^{2n-2}\,{\rm d}u\\
&=\frac{1}{\sqrt{2\pi}}
\int_{-(n-1/2)\sqrt{t}}^{\infty}e^{-u^2/2}(1-e^{-(u+(n-1/2)\sqrt{t})\sqrt{t}})^{2n-2}\,{\rm d}u\\
&+\frac{e^{-(n-1)t}}{\sqrt{2\pi}}
\int_{-(n-3/2)\sqrt{t}}^{\infty}e^{-u^2/2}(1-e^{-(u+(n-3/2)\sqrt{t})\sqrt{t}})^{2n-2}\,{\rm d}u\\
&=N_1(t)+N_2(t).
\end{split}
\end{equation}

For $N_1(t)$, 
\begin{equation*}
\begin{split}
&N_1(t)-\Phi(x)\\
&=\frac{1}{\sqrt{2\pi}}
\int_{-(n-1/2)\sqrt{t}}^{\infty}e^{-u^2/2}
\left\{(1-e^{-(u+(n-1/2)\sqrt{t})\sqrt{t}})^{2n-2}-1\right\}\,{\rm d}u
-\frac{1}{\sqrt{2\pi}}\int_x^{-(n-1/2)\sqrt{t}}e^{-u^2/2}\,{\rm d}u.
\end{split}
\end{equation*}
If $n=1$, then the first term of the right hand side above vanishes. 
For $n\ge 2$, it follows by \eqref{eq:p-ineq} that  for any $t\ge 1$,
\begin{equation}\label{eq:n_1}
\begin{split}
0
&\le \int_{-(n-1/2)\sqrt{t}}^{\infty}e^{-u^2/2}
\left\{1-(1-e^{-(u+(n-1/2)\sqrt{t})\sqrt{t}})^{2n-2}\right\}\,{\rm d}u\\
&\lesssim \int_{-(n-1/2)\sqrt{t}}^{\infty}e^{-u^2/2}
e^{-(u+(n-1/2)\sqrt{t})\sqrt{t}}\,{\rm d}u
=e^{-(n-1)t}\int_{-(n-1/2)\sqrt{t}}^{\infty}e^{-(u+\sqrt{t})^2/2}\,{\rm d}u\\
&= e^{-(n-1)t}\int_{-(n-3/2)\sqrt{t}}^{\infty}e^{-v^2/2}\,{\rm d}v
\lesssim e^{-(n-1)t}.
\end{split}
\end{equation}
Since
$$
\int_x^{-(n-1/2)\sqrt{t}}e^{-u^2/2}\,{\rm d}u
\le \int_{-\infty}^{-(n-1/2)\sqrt{t}}e^{-u^2/2}\,{\rm d}u
\asymp \frac{1}{\sqrt{t}}e^{-(n-1/2)^2t/2},
$$
we have  for any $t\ge 1$,
$$
\sup_{x\le -(n-1/2)\sqrt{t}}\left|N_1(t)-\Phi(x)\right|
\lesssim 
\begin{cases}
e^{-t/8}/\sqrt{t} & (n=1),\\
e^{-(n-1)t} & (n\ge 2).
\end{cases}
$$
For $N_2(t)$, it follows by a calculation similar to \eqref{eq:n_1} that for any $t\ge 1$,
$$
N_2(t)
\lesssim e^{-(n-1)t}\int_{-(n-3/2)\sqrt{t}}^{\infty}e^{-u^2/2}\,{\rm d}u
\lesssim 
\begin{cases}
e^{-t/8}/\sqrt{t} & (n=1),\\
e^{-(n-1)t} & (n\ge 2).
\end{cases}
$$
Combining the two inequalities above with \eqref{eq:k2-bound-11}, 
we get for any $t\ge 1$,
$$
\sup_{x\le -(n-1/2)\sqrt{t}}
\left|a_n(t)K_2(t,x)-\Phi(x)\right|
\lesssim 
\begin{cases}
e^{-t/8}/\sqrt{t} & (n=1),\\
e^{-(n-1)t} & (n\ge 2).
\end{cases}
$$
Hence by \eqref{eq:ii-divide} and \eqref{eq:k1-bound}, we obtain for any $t\ge 1$,
$$
\sup_{x\le -(n-1/2)\sqrt{t}}\left|J_2(t,x)-\Phi(x)\right|
\lesssim 
\frac{1}{\sqrt{t}}.
$$
The proof is complete.
\end{proof}

\begin{lem}\label{lem:even}
For $d=2$, there exists a constant $c>0$ such that 
$$
P\left(\frac{R_t^{(2)}-t/2}{\sqrt{t}}\geq 0\right)
-\Phi(0)\ge \frac{c}{\sqrt{t}}, 
\quad t\geq 1.
$$
\end{lem}

\begin{proof}
Let $R_t=R_t^{(2)}$. 
Then by \eqref{eq:d=2}, 
\begin{equation}\label{eq:dist-d=2}
\begin{split}
&P\left(\frac{R_t-t/2}{\sqrt{t}}\ge 0\right)
=P\left(R_t\ge \frac{t}{2}\right)
=\omega_2\int_{t/2}^{\infty}q_2(t,r)\sinh r\,{\rm d}r\\
&=\frac{e^{-t/8}}{\pi^{1/2}t^{3/2}}
\int_{t/2}^{\infty}\left(\int_r^{\infty}\frac{se^{-s^2/(2t)}}{(\cosh s-\cosh r)^{1/2}}\,{\rm d}s\right)\sinh r\,{\rm d}r
=\frac{e^{-t/8}}{\pi^{1/2}t^{3/2}}I(t).
\end{split}
\end{equation}
By the Fubini theorem, 
\begin{equation*}
\begin{split}
I(t)
&=\int_{t/2}^{\infty}se^{-s^2/(2t)}\left(\int_{t/2}^s\frac{\sinh r}{(\cosh s-\cosh r)^{1/2}}\,{\rm d}r\right)\,{\rm d}s\\
&=2\int_{t/2}^{\infty}se^{-s^2/(2t)}(\cosh s-\cosh (t/2))^{1/2}\,{\rm d}s.
\end{split}
\end{equation*}
Then by the integration by parts formula and the change of variables formula with $s=u\sqrt{t}+t/2$,   
\begin{equation*}
\begin{split}
&I(t)
=t\int_{t/2}^{\infty}e^{-s^2/(2t)}\frac{\sinh s}{(\cosh s-\cosh (t/2))^{1/2}}\,{\rm d}s\\
&=t^{3/2}e^{-t/8}\int_0^{\infty}e^{-u^2/2-u\sqrt{t}/2}
\frac{\sinh(u\sqrt{t}+t/2)}{(\cosh (u\sqrt{t}+t/2)-\cosh(t/2))^{1/2}}\,{\rm d}u\\
&=\frac{t^{3/2}e^{t/8}}{2^{1/2}}\int_0^{\infty}
e^{-u^2/2}
\left(1-\frac{\cosh(t/2)}{\cosh\left(u\sqrt{t}+t/2\right)}\right)^{-1/2}
\left(1-e^{-2(u\sqrt{t}+t/2)}\right)\left(1+e^{-2(u\sqrt{t}+t/2)}\right)^{-1/2}\,{\rm d}u. 
\end{split}
\end{equation*}
Hence by \eqref{eq:dist-d=2},
\begin{equation*}
\begin{split}
&P\left(\frac{R_t-t/2}{\sqrt{t}}\ge 0\right)\\
&=\frac{1}{\sqrt{2\pi }}
\int_0^{\infty}
e^{-u^2/2}
\left(1-\frac{\cosh(t/2)}{\cosh\left(u\sqrt{t}+t/2\right)}\right)^{-1/2}
\left(1-e^{-2(u\sqrt{t}+t/2)}\right)\left(1+e^{-2(u\sqrt{t}+t/2)}\right)^{-1/2}\,{\rm d}u.
\end{split}
\end{equation*}
In particular, if we let 
$$
F_t(u)=\left(1-\frac{\cosh(t/2)}{\cosh\left(u\sqrt{t}+t/2\right)}\right)^{-1/2}
\left(1-e^{-2(u\sqrt{t}+t/2)}\right)\left(1+e^{-2(u\sqrt{t}+t/2)}\right)^{-1/2}-1,
$$
then 
\begin{equation}\label{eq:sharp-1}
P\left(\frac{R_t-t/2}{\sqrt{t}}\ge 0\right)-\Phi(0)
=\frac{1}{\sqrt{2\pi}}\int_0^{\infty}
e^{-u^2/2}F_t(u)\,{\rm d}u. 
\end{equation}

For any $u\ge 0$, 
\begin{equation*}
\begin{split}
&\left(1-e^{-2(u\sqrt{t}+t/2)}\right)^2
-\left(1-\frac{\cosh(t/2)}{\cosh\left(u\sqrt{t}+t/2\right)}\right)
\left(1+e^{-2(u\sqrt{t}+t/2)}\right)\\
&=\frac{\cosh(t/2)}{\cosh\left(u\sqrt{t}+t/2\right)}
\left(1+e^{-2(u\sqrt{t}+t/2)}\right)
-3e^{-2(u\sqrt{t}+t/2)}+e^{-4(u\sqrt{t}+t/2)}\\
&
\ge \frac{e^{-u\sqrt{t}}}{2}-3e^{-2(u\sqrt{t}+t/2)}
=\frac{e^{-u\sqrt{t}}}{2}(1-6e^{-(u\sqrt{t}+t)}).
\end{split}
\end{equation*}
Therefore, if $t\ge \log 6$, then the last term above is positive for any $u\ge 0$ 
so that $F_t(u)$ and the integral in \eqref{eq:sharp-1} are positive.

For $u\geq 0$, we let 
\begin{equation}\label{eq:sharp-2}
\begin{split}
F_t(u)
&=\left(1-\frac{\cosh(t/2)}{\cosh\left(u\sqrt{t}+t/2\right)}\right)^{-1/2}-1\\
&+\left(1-\frac{\cosh(t/2)}{\cosh\left(u\sqrt{t}+t/2\right)}\right)^{-1/2}
\left\{\left(1-e^{-2(u\sqrt{t}+t/2)}\right)\left(1+e^{-2(u\sqrt{t}+t/2)}\right)^{-1/2}-1\right\}.
\end{split}
\end{equation}
Since 
$$
\frac{1}{\sqrt{1-x}}-1\ge 1-\sqrt{1-x}\ge \frac{1}{2}x, \quad 0\le x<1,
$$
we obtain 
\begin{equation}\label{eq:sharp-3}
\left(1-\frac{\cosh(t/2)}{\cosh\left(u\sqrt{t}+t/2\right)}\right)^{-1/2}-1
\ge \frac{1}{2}\frac{\cosh(t/2)}{\cosh\left(u\sqrt{t}+t/2\right)}
\ge \frac{e^{-u\sqrt{t}}}{4}.
\end{equation}

Suppose that $u\ge 1/\sqrt{t}$. Then 
$$
1-\frac{\cosh(t/2)}{\cosh(u\sqrt{t}+t/2)}\ge 1-\frac{\cosh(t/2)}{\cosh(1+t/2)}\ge 1-\frac{2}{e}>0.
$$
Moreover, since 
$$
\left(1-e^{-2(u\sqrt{t}+t/2)}\right)\left(1+e^{-2(u\sqrt{t}+t/2)}\right)^{-1/2}-1\le 0,
$$
we have 
\begin{equation}\label{eq:sharp-4}
\begin{split}
&\left(1-\frac{\cosh(t/2)}{\cosh\left(u\sqrt{t}+t/2\right)}\right)^{-1/2}
\left\{\left(1-e^{-2(u\sqrt{t}+t/2)}\right)\left(1+e^{-2(u\sqrt{t}+t/2)}\right)^{-1/2}-1\right\}\\
&=-\left(1-\frac{\cosh(t/2)}{\cosh\left(u\sqrt{t}+t/2\right)}\right)^{-1/2}
\left\{1-\left(1-e^{-2(u\sqrt{t}+t/2)}\right)\left(1+e^{-2(u\sqrt{t}+t/2)}\right)^{-1/2}\right\}\\
&\ge -(1-2e^{-1})^{-1/2}\left\{1-\left(1-e^{-2(u\sqrt{t}+t/2)}\right)\left(1+e^{-2(u\sqrt{t}+t/2)}\right)^{-1/2}\right\}.
\end{split}
\end{equation}
Noting that 
$$
1-\frac{1}{\sqrt{1+t}}\asymp t, \quad 0\le t\le 1,
$$
we obtain 
$$
0\le 1-(1+e^{-2(u\sqrt{t}+t/2)})^{-1/2}\le c_1e^{-2(u\sqrt{t}+t/2)}
$$
so that 
\begin{equation}\label{eq:sharp-5}
\begin{split}
&1-\left(1-e^{-2(u\sqrt{t}+t/2)}\right)\left(1+e^{-2(u\sqrt{t}+t/2)}\right)^{-1/2}\\
&=e^{-2(u\sqrt{t}+t/2)}
+\left(1-e^{-2(u\sqrt{t}+t/2)}\right)\left\{1-\left(1+e^{-2(u\sqrt{t}+t/2)}\right)^{-1/2}\right\}
\le c_2e^{-2(u\sqrt{t}+t/2)}.
\end{split}
\end{equation}
Therefore, by \eqref{eq:sharp-2}, \eqref{eq:sharp-3}, \eqref{eq:sharp-4} and \eqref{eq:sharp-5}, 
we have for any $t\ge 1$ and $u\ge 1/\sqrt{t}$, 
\begin{equation}\label{eq:f-lower}
F_t(u)\ge \frac{1}{4}e^{-u\sqrt{t}}-c_2e^{-2(u\sqrt{t}+t/2)}.
\end{equation}

Note that  for any $c>0$, 
$$
\int_{1/\sqrt{t}}^{\infty}e^{-u^2/2-cu\sqrt{t}}\,{\rm d}u
=e^{c^2t/2}\int_{1/\sqrt{t}}^{\infty}e^{-(u+c\sqrt{t})^2/2}\,{\rm d}u
=e^{c^2 t/2}\int_{c\sqrt{t}+1/\sqrt{t}}^{\infty}e^{-u^2/2}\,{\rm d}u
\asymp \frac{1}{\sqrt{t}}, \quad t\ge 1.
$$
Hence by \eqref{eq:f-lower} and taking $c=1$ and $c=2$ in the relation above, 
we get 
\begin{equation*}
\begin{split}
\int_0^{\infty}e^{-u^2/2}F_t(u)\,{\rm d}u
\ge \int_{1/\sqrt{t}}^{\infty}
e^{-u^2/2}F_t(u)\,{\rm d}u
&\ge \frac{1}{4}\int_{1/\sqrt{t}}^{\infty}e^{-u^2/2-u\sqrt{t}}\,{\rm d}u
-c_2e^{-t}\int_{1/\sqrt{t}}^{\infty}e^{-u^2/2-2u\sqrt{t}}\,{\rm d}u\\
&\ge \frac{c_3}{\sqrt{t}}.
\end{split}
\end{equation*}
Combining this with \eqref{eq:sharp-1}, 
we arrive at the desired conclusion.
\end{proof}

\begin{proof}[Proof of Theorem {\rm \ref{thm:berry}} for even dimensions]
The proof is complete by Lemmas \ref{lem:i-bound}, \ref{lem:ii-bound} and \ref{lem:even}.
\end{proof}

\appendix
\section{Appendix}\label{sect:append}
To prove Theorem \ref{thm:berry} for general odd dimensions, 
we show 
\begin{lem}\label{lem:deri}
For $l\ge 1$,
\begin{equation}\label{eq:deri-0}
\left(\frac{1}{\sinh r}\frac{\partial}{\partial r}\right)^l\sinh^{2l+1}r
=\frac{(2l+1)!!}{l+1}\sinh((l+1)r), \ r\in {\mathbb R}.
\end{equation}
\end{lem}

\begin{proof}
Fix $n\geq 2$. We first prove by induction that for any $k=1,2,3,\dots, [n/2]$,
\begin{equation}\label{eq:ind-1}
\left(\frac{1}{\sinh r}\frac{\partial}{\partial r}\right)^{2k}\sinh^n r
=\sum_{l=0}^k\frac{(2k)!}{2^{k-l}(k-l)!(2l)!}
\left(\prod_{m=0}^{k+l-1}(n-2m)\right)\cosh^{2l}r\sinh^{n-(2k+2l)}r, 
\ r\in {\mathbb R}.
\end{equation}
For $k=1$, we have 
$$
\left(\frac{1}{\sinh r}\frac{\partial}{\partial r}\right)^2\sinh^n r
=n\sinh^{n-2}r+n(n-2) \cosh^2 r \sinh^{n-4}r
$$
so that \eqref{eq:ind-1} is valid. 
Suppose that \eqref{eq:ind-1} holds for some $k\geq 1$. 
Since 
\begin{equation*}
\begin{split}
&\frac{\partial}{\partial r}(\cosh^{2l}r\sinh^{n-(2k+2l)}r)\\
&=2l\cosh^{2l-1}r\sinh^{n-(2k+2l-1)}r+(n-(2k+2l))\cosh^{2l+1}r\sinh^{n-(2k+2l+1)}r,
\end{split}
\end{equation*}
we obtain by \eqref{eq:ind-1}, 
\begin{equation}\label{eq:ind-00}
\begin{split}
&\frac{\partial}{\partial r}\left(\frac{1}{\sinh r}\frac{\partial}{\partial r}\right)^{2k}\sinh^n r\\
&=\sum_{l=0}^k\frac{(2k)!}{2^{k-l}(k-l)!(2l)!}
\left(\prod_{m=0}^{k+l-1}(n-2m)\right)\frac{\partial}{\partial r}(\cosh^{2l}r\sinh^{n-(2k+2l)}r)\\
&=\sum_{l=1}^k\frac{(2k)!}{2^{k-l}(k-l)!(2l-1)!}
\left(\prod_{m=0}^{k+l-1}(n-2m)\right)\cosh^{2l-1}r\sinh^{n-(2k+2l-1)}r\\
&+\sum_{l=0}^k \frac{(2k)!}{2^{k-l}(k-l)!(2l)!}
\left(\prod_{m=0}^{k+l}(n-2m)\right)\cosh^{2l+1}r\sinh^{n-(2k+2l+1)}r\\
&=\sum_{l=1}^k\frac{(2k)!}{2^{k-l}(k-l)!(2l-1)!}
\left(\prod_{m=0}^{k+l-1}(n-2m)\right)\cosh^{2l-1}r\sinh^{n-(2k+2l-1)}r\\
&+\sum_{l=1}^{k+1} \frac{(2k)!}{2^{k-l+1}(k-l+1)!(2l-2)!}
\left(\prod_{m=0}^{k+l-1}(n-2m)\right)\cosh^{2l-1}r\sinh^{n-(2k+2l-1)}r\\
&=\sum_{l=1}^{k+1} \frac{(2k+1)!}{2^{k+1-l}(k+1-l)!(2l-1)!}
\left(\prod_{m=0}^{k+l-1}(n-2m)\right)\cosh^{2l-1}r\sinh^{n-(2k+2l-1)}r.
\end{split}
\end{equation}
Therefore, 
\begin{equation}\label{eq:ind-2}
\begin{split}
&\left(\frac{1}{\sinh r}\frac{\partial}{\partial r}\right)^{2k+1}\sinh^n r\\
&=\sum_{l=1}^{k+1} \frac{(2k+1)!}{2^{k+1-l}(k+1-l)!(2l-1)!}
\left(\prod_{m=0}^{k+l-1}(n-2m)\right)\cosh^{2l-1}r\sinh^{n-(2k+2l)}r.
\end{split}
\end{equation}
Using this equality, we obtain in a similar way to \eqref{eq:ind-00}, 
\begin{equation*}
\begin{split}
&\frac{\partial}{\partial r}\left(\frac{1}{\sinh r}\frac{\partial}{\partial r}\right)^{2k+1}\sinh^n r\\
&=\sum_{l=0}^{k+1} \frac{(2k+2)!}{2^{k+1-l}(k+1-l)!(2l)!}
\left(\prod_{m=0}^{k+l}(n-2m)\right)\cosh^{2l}r\sinh^{n-(2k+2l+1)}r
\end{split}
\end{equation*}
and thus
\begin{equation*}
\begin{split}
&\left(\frac{1}{\sinh r}\frac{\partial}{\partial r}\right)^{2k+2}\sinh^n r\\
&=\sum_{l=0}^{k+1} \frac{(2k+2)!}{2^{k+1-l}(k+1-l)!(2l)!}
\left(\prod_{m=0}^{k+l}(n-2m)\right)\cosh^{2l}r\sinh^{n-(2k+2+2l)}r.
\end{split}
\end{equation*}
This shows that \eqref{eq:ind-1} is valid for $k$ replaced with $k+1$, 
that is, the induction is complete.
In particular, our argument above implies that 
\eqref{eq:ind-1} and \eqref{eq:ind-2} hold for any $k=1,2,3,\dots, [n/2]$.

We next prove \eqref{eq:deri-0}. 
We here calculate the left hand side of \eqref{eq:deri-0} 
for  odd $l$ and even $l$, respectively. 
Namely, we will prove that for any $k\geq 1$,
\begin{equation}\label{eq:deri-1}
\left(\frac{1}{\sinh r}\frac{\partial}{\partial r}\right)^{2k-1}\sinh^{4k-1}r=\frac{(4k-1)!!}{2k}\sinh 2kr
\end{equation}
and 
\begin{equation}\label{eq:deri-2}
\left(\frac{1}{\sinh r}\frac{\partial}{\partial r}\right)^{2k}\sinh^{4k+1}r
=\frac{(4k+1)!!}{2k+1}\sinh ((2k+1)r).
\end{equation}

By \eqref{eq:ind-2}, 
\begin{equation*}
\begin{split}
&\left(\frac{1}{\sinh r}\frac{\partial}{\partial r}\right)^{2k-1}\sinh^{4k-1} r\\
&=\sum_{l=1}^k \frac{(2k-1)!}{2^{k-l}(k-l)!(2l-1)!}
\left(\prod_{m=0}^{k+l-2}(4k-1-2m)\right)\cosh^{2l-1}r\sinh^{2k-2l+1}r\\
&=\sum_{j=0}^{k-1} \frac{(2k-1)!}{2^j j! (2k-2j-1)!}
\left(\prod_{m=0}^{2k-j-2}(4k-1-2m)\right)\cosh^{2k-(2j+1)}r\sinh^{2j+1}r,
\end{split}
\end{equation*}
where we set $j=k-l$ in the second equality above. 
If $0\leq j\leq k-1$, then  
\begin{equation*}
\begin{split}
&\frac{(2k-1)!}{2^j j! (2k-2j-1)!}\prod_{m=0}^{2k-j-2}(4k-1-2m)
=\frac{(2k-1)!}{(2j)!!(2k-2j-1)!}\prod_{l=j+2}^{2k}(2l-1)\\
&=\frac{(4k-1)!!}{2k}\frac{(2k)!}{(2j+1)!(2k-(2j+1))!}
=\frac{(4k-1)!!}{2k}\binom{2k}{2j+1}
\end{split}
\end{equation*}
and thus
\begin{equation*}
\begin{split}
\left(\frac{1}{\sinh r}\frac{\partial}{\partial r}\right)^{2k-1}\sinh^{4k-1} r
&=\frac{(4k-1)!!}{2k}\sum_{j=0}^{k-1} \binom{2k}{2j+1}\cosh^{2k-(2j+1)}r\sinh^{2j+1}r\\
&=\frac{(4k-1)!!}{2k}\sinh 2kr.
\end{split}
\end{equation*}
The proof of \eqref{eq:deri-1} is complete.

By \eqref{eq:ind-1} and a similar argument as above, 
\begin{equation*}
\begin{split}
&\left(\frac{1}{\sinh r}\frac{\partial}{\partial r}\right)^{2k}\sinh^{4k+1} r\\
&=\sum_{l=0}^k\frac{(2k)!}{2^{k-l}(k-l)!(2l)!}
\left(\prod_{m=0}^{k+l-1}(4k+1-2m)\right)\cosh^{2l}r\sinh^{2k-(2l-1)}r\\
&=\sum_{j=0}^k\frac{(2k)!}{2^j j!(2(k-j))!}
\left(\prod_{i=j+1}^{2k}(2i+1)\right)\cosh^{2(k-j)}r\sinh^{2j+1}r\\
&=\frac{(4k+1)!!}{2k+1}
\sum_{j=0}^k \binom{2k+1}{2j+1}\cosh^{2(k-j)}r\sinh^{2j+1}r
=\frac{(4k+1)!!}{2k+1}\sinh((2k+1)r).
\end{split}
\end{equation*}
The proof of  \eqref{eq:deri-2} is  complete.
\end{proof}

\end{document}